\documentclass[12pt]{amsart}
\usepackage{amssymb}
\usepackage{amsthm}
\usepackage{amsmath}
\usepackage[T1]{fontenc}
\usepackage[latin1]{inputenc}
\usepackage[all]{xy}
\usepackage{verbatim}
\usepackage{vmargin}
\usepackage{array}
\setpapersize{USletter}
\setmargrb{1in}{1in}{1in}{1in}
\DeclareMathOperator{\Span}{span}

\DeclareMathOperator{\rk}{rk}

\DeclareMathOperator{\Pic}{Pic}
\DeclareMathOperator{\Jac}{Jac}
\DeclareMathOperator{\codim}{codim}

\DeclareMathOperator{\Sec}{Sec}
\DeclareMathOperator{\td}{td}
\DeclareMathOperator{\ch}{ch}

\newtheorem{thm}{Theorem}[section]
\newtheorem{prop}[thm]{Proposition}

\newtheorem{defn}[thm]{Definition}

\newtheorem{lemma}[thm]{Lemma}

\title{The degree of the third secant variety of a smooth curve of genus $2$}
\author{Andrea Hofmann}
\address{Matematisk Institutt\\
Universitetet I Oslo\\
PO Box 1053\\
Blindern, NO-0316 Oslo, Norway}
\email{andrehof@math.uio.no}

\begin{document}

\begin{abstract}
We give a new method of computation of the degree of the third secant variety $\Sec_3(C)$ of a smooth curve $C\subseteq \mathbf P^{d-2}$ of genus $2$ and degree $d\geq 8$, using the presentation of $\Sec_3(C)$ as the union of all scrolls that are defined via a $g^1_3$ on $C$.
\end{abstract}

\maketitle

\section{Introduction}\label{intro}

\noindent Berzolari's formula from 1895 (cf. \cite{BallCoss}, Section $4$) computes the number of trisecant lines to a smooth curve of genus $g$ and degree $d$ in $\mathbf P^4$, in the case this number is finite. This number is equal to $\binom{d-2}{3}-g(d-4)$.

\noindent In this paper $C$ denotes a smooth and irreducible curve of genus $2$ and degree $d\geq 8$ embedded in $\mathbf P^{d-2}$.\\
The third secant variety of $C$, $\Sec_3(C)$, is defined as the closure of the union of all trisecant planes to $C$:
$$\Sec_3(C)=\overline{\bigcup_{D\in C_3}\Span(D)},$$

\noindent where $C_3:=(C\times C\times C)/S_3$ parametrizes effective divisors of degree $3$ on $C$, and $\Span(D)$ denotes the plane spanned by the three points in $D$.

\noindent The dimension of $\Sec_3(C)$ is equal to $\dim(C_3)+\dim(\Span(D))=5$, i.e. in order to find the degree of $\Sec_3(C)$ we have to intersect with five general hyperplanes.\\
Let now $V$ denote the intersection of five general hyperplanes in $\mathbf P^{d-2}$, i.e. $V$ is a general space of codimension $5$ in $\mathbf P^{d-2}$.
Since $\dim(\Sec_2(C))=3$ and $\codim(V)=5$, $V$ and $\Sec_2(C)$ do not intersect. This implies that $V$ cannot intersect any trisecant plane to $C$ in a line, since every trisecant plane to $C$ contains three lines in $\Sec_2(C)$, and so if $V$ intersects a trisecant plane to $C$ in a line $L$, then $L$ intersects at least one of those lines in $\Sec_2(C)$ in a point which obviously lies in $\Sec_2(C)$.\\
Projecting from $V$ to $\mathbf P^4$ gives us the equality of the degree of $\Sec_3(C)$ and the number of trisecant lines to a curve $C\subseteq \mathbf P^4$ of genus $2$ and the same degree $d$ in the following way: Since $V$ was chosen to be a general space of codimension $5$, $V$ does not intersect the curve $C$, and thus the image of $C$ under the projection from $V$ is a curve of degree $d$ and genus $2$ in $\mathbf P^4$. Moreover, the fact that $V$ does not intersect $\Sec_2(C)$ implies that the image curve is smooth as well.

\noindent A trisecant plane to $C\subseteq \mathbf P^{d-2}$ which intersects $V$ in one point projects down to a trisecant line to the image curve in $\mathbf P^4$.

\noindent Summarizing, the number of trisecant planes to $C\subseteq \mathbf P^{d-2}$ that intersect $V$ in one point is equal to the number of trisecant lines to the image curve in $\mathbf P^4$, and thus it follows that the degree of $\Sec_3(C)$ is equal to the number of trisecant lines to the image curve in $\mathbf P^4$.\\
Consequently, the degree of $\Sec_3(C)$ is equal to $\binom{d-2}{3}-2(d-4)$, and our motivation is now to compute the degree of $\Sec_3(C)$ in a different way, identifying $\Sec_3(C)$ as the union of all scrolls defined via a $g^1_3$ on $C$. 

\noindent Any abstract curve $C$ of genus $2$ can be embedded as a smooth curve of degree $d\geq 5$ into $\mathbf P^{d-2}$.\\
In this paper we restrict ourselves to the case $d\geq 8$, since for a curve $C$ of genus $2$ and degree $d=6$ or $d=7$, although Berzolari's formula of course being valid, taking Berzolari's formula to compute the degree of $\Sec_3(C)$ does not make sense, since for these values of $d$ the third secant variety $\Sec_3(C)$ is equal to the ambient space $\mathbf P^{d-2}$.
For $d=5$ the following holds: There are infinitely many trisecant lines to a curve $C\subseteq \mathbf P^3$ of genus $2$ and degree $5$, since $C$ lies on a quadric on which there exists a one-dimensional family of lines that each intersects $C$ in three points.

\section{Preliminaries}\label{preliminaries}

\noindent Let $C$ be a smooth curve of genus $2$ and degree $d\geq 8$ embedded in projective space $\mathbf P^{d-2}$.

\noindent For each $g^1_3$ on $C$, which we denote by $\vert D\vert$, we set 

$$V_{\vert D\vert}:=\overline{\bigcup_{D^{\prime}\in \vert D\vert}\Span(D^{\prime})},$$

\noindent where $\Span D^{\prime}$ denotes the plane spanned by the three points in $\vert D\vert$.

\noindent Each $V_{\vert D\vert}$ is a threedimensional rational normal scroll. (For general theory about rational normal scrolls we refer to \cite{EisenHarris}.)

\noindent We set $G^1_3(C):=\{g^1_3\textrm{'s on }C\}$. Since our aim is to identify $\Sec_3(C)=\bigcup_{\vert D\vert\in G^1_3(C)}V_{\vert D\vert}$, we want to find the dimension of $\bigcup_{\vert D\vert\in G^1_3(C)}V_{\vert D\vert}$, and for this purpose we need the dimension of the family $G^1_3(C)$, which we will now compute:

\begin{prop}\label{dimg13}
Let $C$ be a curve of genus $2$. The family $G^1_3(C)=\{g^1_3\textrm{'s on }C\}$ is two-dimensional.
\end{prop}

\begin{proof}
If $D$ is a divisor of degree $3$ on $C$, then $h^0(\mathcal O_C(D))=2$ by the Riemann-Roch theorem for curves (see e.g. \cite{Hartshorne}, Thm. 1.3 in Chapter IV.1), i.e. each linear system $\vert D\vert$ of degree $3$ is a $g^1_3$ on $C$. The set of all effective divisors of degree $3$ on $C$ is given by $C_3:=(C\times C\times C)/S_3$, where $S_3$ denotes the symmetric group on $3$ letters. The dimension of this family is equal to $3$, and since each linear system $\vert D\vert$ of degree $3$ on $C$ has dimension $1$, as shown above, the family of $g^1_3$'s on $C$ has to be two-dimensional.
\end{proof}

\noindent We obtain that the dimension of $\bigcup_{\vert D\vert\in G^1_3(C)}V_{\vert D\vert}$ is equal to $5$, which is also the dimension of $\Sec_3(C)$, as we have seen in the introduction, and since each scroll $V_{\vert D\vert}$ obviously is contained in $\Sec_3(C)$, we obtain equality:
$$\Sec_3(C)=\bigcup_{\vert D\vert\in G^1_3(C)}V_{\vert D\vert}.$$

\noindent For an integer $k\geq 0$ we denote by $\Pic^k(C)$ the set of all line bundles of degree $k$ on $C$ modulo isomorphism.
In this paper we will consider $k=0$ and $k=3$.\\
We use the definition of the Jacobian variety of $C$, $\Jac(C)$, as in \cite{Milne}, namely that $\Jac(C)$ is defined as the abelian variety that represents the functor $T\to Pic^0(C/T)$ from schemes over the base field $k$ to abelian groups (cf. \cite{Milne}, Theorem 1.1).

\noindent By fixing a divisor $D_0$ of degree $3$ we obtain an isomorphism 

\vspace{-0.1cm}

$$\mu: \Pic^0(C)\to \Pic^3(C),$$
$$[\mathcal O_C(D)]\mapsto [\mathcal O_C(D+D_0)].$$

\vspace{0.1cm}

\noindent Hence $\Pic^3(C)$ is isomorphic to the Jacobian variety $\Jac(C)$.

\noindent Fixing a point $P_0$ on $C$ gives an embedding

\vspace{-0.1cm}

$$\nu:C\to \Jac(C),$$
$$R\mapsto [\mathcal O_C(R-P_0)].$$

\vspace{0.1cm}

\noindent The dimension of $\Jac(C)$ is equal to the genus of $C$, which is equal to $2$. Hence $\Jac(C)$ is an abelian surface. The theta divisor $\Theta$ on $\Jac(C)$ is the image of $C$ under the above map $\nu$.\\
For fixed points $P$ and $Q$ on $C$ we define 
$$\Theta_{P,Q}:=\{[\mathcal O_C(P+Q+R)]\vert R\in C\}.$$

\noindent $\Theta_{P,Q}$ is a divisor on $\Pic^3(C)$, and using the above isomorphism $\mu$ with $D_0=P+Q+P_0$ we see that the divisor $\Theta_{P,Q}$ is isomorphic to $\Theta$. It is this $\Theta_{P,Q}$ we will use in Sections \ref{chern} and \ref{degsec} when we consider $\Theta$ on $\Pic^3(C)$.

\begin{prop}\label{theta2}
The divisor $\Theta$ has self-intersection $\Theta^2=2$.
\end{prop}

\begin{proof}
Choose points $P$, $P^{\prime}$, $Q_1$ and $Q_2$, $Q_1\neq Q_2$, on $C$ such that $P+P^{\prime}$ is a divisor in the canonical system on $C$, $\vert K_C\vert$, and such that $Q_1+Q_2$ is not a divisor in $\vert K_C\vert$. There exist points $Q_1^{\prime}$ and $Q_2^{\prime}$ on $C$ such that $Q_1+Q_1^{\prime}\in \vert K_C\vert$ and $Q_2+Q_2^{\prime}\in \vert K_C\vert$. We obtain the following:

\vspace{-0.5cm}

\begin{eqnarray*}
\Theta^2&=&\Theta_{P,Q_1}.\Theta_{P^{\prime},Q_2}\\
&=&\#\{[\mathcal O_C(Q_1+Q_2+R)]\vert R\in \{Q_1^{\prime},Q_2^{\prime}\}\}\\
&=&2.
\end{eqnarray*}
\end{proof}

\noindent Consider now the following projections:

\vspace{0.4cm}

\begin{tabular}{lccr}
\xymatrix{
&C\times \Pic^3(C)\ar[dl]_p\ar[dr]^q&\\
C&&\Pic^3(C)\\
}
&
&&
\xymatrix{
&C\times C_3\ar[dl]\ar[dr]^{\pi}&\\
C&&C_3\hspace{0.1cm}\\
}\\
\end{tabular}

\vspace{0.3cm}

\noindent Let $P$ be a point on $C$ such that $2P$ is a divisor in the canonical system $\vert K_C\vert$, and set $f:=p^{\ast}(P)$.

\noindent In the rest of this paper we will use the notation $P$ and $f$ both as varieties and as classes.

\noindent As before, we define $C_3$ to be the threedimensional family of all effective divisors of degree $3$ on $C$.

\noindent Let $\Delta$ be the universal divisor on $C\times C_3$, i.e. $\Delta\vert_{C\times \{D\}}\cong D$ for all $D\in C_3$.\\
For any point $Q$ on $C$ set $X_Q:=\{D\in C_3\vert Q\in D\}$, which is a divisor on $C_3$.\\ 
Finally, let $u:C_3\to \mathcal \Pic^3(C)$ be the map given by $u(D):=[\mathcal O_C(D)]$.

\noindent Now we are able to define a line bundle $\mathcal L$ on $C\times \Pic^3(C)$ which turns out to be a Poincaré line bundle. In Section \ref{degsec} we will compute the degree of $\Sec_3(C)$ by identifying $\Sec_3(C)$ as a degeneracy locus of a map of vector bundles involving this Poincaré line bundle $\mathcal L$.

\section{The Poincaré line bundle $\mathcal L$}\label{poincare}

\noindent We will first give the definition of a Poincaré line bundle:

\begin{defn}
A \textup{Poincaré line bundle of degree $k$} is a line bundle $\mathcal L$ on $C\times \Pic^k(C)$ such that $\mathcal L\vert_{C\times [\mathcal O_{C}(D)]}\cong \mathcal O_C(D)$ for all points $[\mathcal O_C(D)]$ in $\Pic^k(C)$.
\end{defn}

\noindent Set $\mathcal L:=(1\times u)_{\ast}(\mathcal O_{C\times C_3}(\Delta-\pi^{\ast}(X_Q)))$.
$\mathcal L$ is a Poincaré line bundle of degree $3$ (cf. \cite{ACGH}, Chapter IV, §2, p. 167).

\noindent Let $\vert H\vert$ be the linear system of degree $d$ that embeds $C$ into projective space $\mathbf P^{d-2}$. Set $\mathcal H:=q_{\ast}(\mathcal L)$ and $\mathcal G:=q_{\ast}(p^{\ast}\mathcal O_C(H)\otimes \mathcal L^{-1})$.

\noindent Since the fiber of $\mathcal H$ over a point $[\mathcal O_C(D)]\in \mathcal \Pic^3(C)$ is equal to $H^0(\mathcal O_C(D))$, the rank of $\mathcal H$ is equal to $h^0(\mathcal O_C(D))=2$, and since the fiber of $\mathcal G$ over a point $[\mathcal O_C(D)]\in \mathcal \Pic^3(C)$ is equal to $H^0(\mathcal O_C(H-D))$, the rank of $\mathcal G$ is equal to $h^0(\mathcal O_C(H-D))=d-4$.\\
We will use these vector bundles $\mathcal H$ and $\mathcal G$ in Section \ref{degsec} to define a map of vector bundles which degeneracy locus is equal to the third secant variety of $C$, $\Sec_3(C)$.
We will need the Chern classes of $\mathcal H$ and $\mathcal G$, and for this purpose we need the Chern classes of $\mathcal L$.
We will find all of these Chern classes in the next section.
 
\section{The Chern classes of $\mathcal L$, $\mathcal H$ and $\mathcal G$}\label{chern}

\noindent In this section we will find the Chern classes of $\mathcal L$, $\mathcal H$ and $\mathcal G$ as defined in Section \ref{poincare}.

\subsection{The Chern class of $\mathcal L$}

\noindent By \cite{ACGH}, Chapter VIII, §2 (pp. 333-336) we obtain that the first Chern class of $\mathcal L$ is equal to 
$c_1(\mathcal L)=3f+\gamma$, where $\gamma$ is the diagonal component of $c_1(\mathcal L)$ in the term $H^1(C)\otimes H^1(\Pic^3(C))$ of the Künneth decomposition

\begin{eqnarray*}
H^2(C\times \Pic^3(C))&=&\hspace{-0.2cm}(H^2(C)\otimes H^0(\Pic^3(C)))\\
&\oplus&\hspace{-0.2cm} (H^1(C)\otimes H^1(\Pic^3(C)))\\
&\oplus& \hspace{-0.2cm}(H^0(C)\otimes H^2(\Pic^3(C))).
\end{eqnarray*}

\vspace{0.1cm}

\noindent The following is satisfied: $\gamma^2=-2f.q^{\ast}(\Theta)$, $\gamma^3=f.\gamma=0$, where now $\Theta$ on $\Pic^3(C)$ is equal to $\Theta_{P,Q}$ as defined in Section \ref{preliminaries}.

\noindent Thus for the Chern character of $\mathcal L$ we obtain:
$$\ch(\mathcal L)=e^{c_1(\mathcal L)}=1+3f+\gamma-f.q^{\ast}(\Theta).$$

\subsection{The Chern classes of $\mathcal H$}

\noindent Recall that we had defined $\mathcal H:=q_{\ast}\mathcal L$. The Chern character of $\mathcal H$ we obtain by the Grothendieck-Riemann-Roch Theorem (cf. \cite{Fulton}, Thm. 15.2):

\vspace{-0.1cm}

$$\ch(q_{\ast}(\mathcal L)).\td(\Pic^3(C))=q_{\ast}(\ch(\mathcal L).\td(C\times \Pic^3(C))).$$

\vspace{0.1cm}

\noindent Before we can continue our computation of $\ch(\mathcal H)$ we need some Todd classes and pushforwards.

\begin{defn}\label{toddclass}(cf. \cite{Fulton}, Example 3.2.4)
The Todd class of a vector bundle $E$ of rank $r$ on a variety $X$ is defined as
$$\td(E)=\prod_{i=1}^r\frac{\alpha_i}{1-e^{\alpha_i}},$$
\noindent where $\alpha_1,\ldots,\alpha_r$ are the Chern roots of $E$.\\
If $Y$ is a variety, then by $\td(Y)$ we denote $\td(T_Y)$, the Todd class of the tangent bundle of $Y$.
\end{defn}

\noindent We will need Todd classes only in the cases when the dimension of $X$ is equal to $1$ or $2$. In these cases $c_i(E)=0$ for $i\geq 3$, and expanding the above product yields:
$$\td(E)=1+\frac{1}{2}c_1(E)+\frac{1}{12}(c_1^2(E)+c_2(E)).$$

\begin{lemma}\label{todd}
We have the following Todd classes:

\begin{itemize}
\item[(1)] $\td(\Pic^3(C))=1$.
\item[(2)] $\td(C)=1-P$.
\item[(3)] $\td(C\times \Pic^3(C))=1-f$.
\end{itemize}
\end{lemma}

\begin{proof}
\noindent\begin{itemize}
\item[(1)] Since $\Pic^3(C)\cong \Jac(C)$ is an abelian variety, we have $K_{\Pic^3(C)}=0$ and thus also $c_1(T_{\Pic^3(C)})=0$.
\item[(2)] $\td(C)=1+\frac{1}{2}c_1(T_C)=1-\frac{1}{2}[K_C]=1-P$.
\item[(3)] $\td(C\times \Pic^3(C))=\td(p^{\ast}(C)).\td(q^{\ast}\Pic^3(C))=1-f$.
\end{itemize}

\vspace{-0.5cm}

\end{proof}

\begin{lemma}\label{push}
We have the following pushforwards:

\begin{itemize}
\item[(1)] $q_{\ast}(1)=0$.
\item[(2)] $q_{\ast}(f)=1$.
\item[(3)] $q_{\ast}(\gamma)=0$.
\end{itemize}
\end{lemma}

\begin{proof}
\noindent \begin{itemize}
\item[(1)] $q_{\ast}(1)=q_{\ast}([C\times \Pic^3(C)])=0$, since $\dim(q(C\times \Pic^3(C)))=\dim(\Pic^3(C))=2<3=\dim(C\times \Pic^3(C))$.
\item[(2)] Since $q(f)=\Pic^3(C)$ has the same dimension as $f$, we have $q_{\ast}(f)=a[\Pic^3(C)]$ for a positive integer $a$. By the projection formula (cf. \cite{Fulton}, Prop. 2.5(c)) we obtain for every point $[\mathcal O_C(D_0)]\in \Pic^3(C)$:

\vspace{-0.8cm}

\begin{eqnarray*}
a&=&q_{\ast}(f).[\mathcal O_C(D_0)]=q_{\ast}(f.q^{\ast}[\mathcal O_C(D_0)])=f.q^{\ast}[\mathcal O_C(D_0)]\\
&=&[P\times \Pic^3(C)].[C\times \mathcal O_C(D_0)]=1,
\end{eqnarray*}

\noindent where we could use the equality $q_{\ast}(f.q^{\ast}[\mathcal O_C(D_0)])=f.q^{\ast}[\mathcal O_C(D_0)]$, since $f.q^{\ast}[\mathcal O_C(D_0)]$ is $0$-dimensional.
\item[(3)] Since $\gamma$ is of codimension $1$ on $C\times \Pic^3(C)$, $q_{\ast}(\gamma)=a[\Pic^3(C)]$ for some non-negative integer $a$. By the projection formula we have for every point $[\mathcal O_C(D_0)]\in \Pic^3(C)$:

\vspace{-0.6cm}

\begin{eqnarray*}
a&=&q_{\ast}(\gamma).[\mathcal O_C(D_0)]=q_{\ast}(\gamma.q^{\ast}[\mathcal O_C(D_0)])=\gamma.q^{\ast}[\mathcal O_C(D_0)]\\
&=&c_1(\mathcal L).q^{\ast}[\mathcal O_C(D_0)]
-q_{\ast}(3f)=3-3=0,
\end{eqnarray*}

\noindent where, analogously to (2), we could use the equality $q_{\ast}(\gamma.q^{\ast}[\mathcal O_C(D_0)])=\gamma.q^{\ast}[\mathcal O_C(D_0)]$ since $\gamma.q^{\ast}[\mathcal O_C(D_0)]$ is $0$-dimensional.
\end{itemize}
\end{proof}

\noindent Now, by Lemma \ref{todd} we obtain:

\vspace{-0.4cm}

\begin{eqnarray*}
\ch(\mathcal H)&=&\ch(q_{\ast}(\mathcal L))\\
&=&q_{\ast}(\ch(\mathcal L).(1-f))=q_{\ast}((1+3f+\gamma-f.q^{\ast}(\Theta)).(1-f))\\
&=&q_{\ast}(1+2f+\gamma-f.q^{\ast}(\Theta)).
\end{eqnarray*}

\vspace{0.1cm}

\noindent By Lemma \ref{push} and the projection formula we can conclude:
$$\ch(\mathcal H)=2-q_{\ast}(f).\Theta=2-\Theta.$$

\vspace{0.1cm}

\noindent Consequently we obtain for the Chern polynomial of $\mathcal H$:

\begin{equation}\label{cpolyH}
c_t(\mathcal H)=e^{-\Theta t}.
\end{equation}

\subsection{The Chern classes of $\mathcal G$}

\noindent Now we want to find the Chern classes of the vector bundle $\mathcal G:=q_{\ast}(p^{\ast}(\mathcal O_C(H)\otimes \mathcal L^{-1}))$, where $\vert H\vert$ denotes the linear system of degree $d$ that embeds $C$ into projective space. In order to do so we use again the Grothendieck-Riemann-Roch formula:
$$\ch(q_{\ast}(p^{\ast}(\mathcal O_C(H)\otimes \mathcal L^{-1}))).\td(\Pic^3(C))=q_{\ast}(\ch(p^{\ast}(\mathcal O_C(H)\otimes \mathcal L^{-1})).\td(C\times \Pic^3(C))).$$

\vspace{0.1cm}

\noindent By Lemma \ref{todd}, Lemma \ref{push} and the projection formula we obtain

\vspace{-0.2cm}

\begin{eqnarray*}
\ch(\mathcal G)&=&q_{\ast}(\ch(p^{\ast}(\mathcal O_C(H)\otimes \mathcal L^{-1}).(1-f))\\
&=&q_{\ast}(p^{\ast}(\ch(\mathcal O_C(H))).\ch(\mathcal L^{-1}).(1-f))\\
&=&q_{\ast}(1+p^{\ast}(H)).(1-3f-\gamma-f.q^{\ast}(\Theta)).(1-f))\\
&=&q_{\ast}((1+df).(1-4f-\gamma-f.q^{\ast}(\Theta)))\\
&=&q_{\ast}(1+(d-4)f-\gamma-f.q^{\ast}(\Theta))\\
&=&d-4-\Theta.
\end{eqnarray*}

\vspace{0.1cm}

\noindent This yields for the Chern polynomial of $\mathcal G$:

\begin{equation}\label{cpolyG}
c_t(\mathcal G)=e^{-\Theta t}.
\end{equation}

\vspace{0.2cm}

\section{The degree of $\Sec_3(C)$}\label{degsec}
\noindent Set $E:=\mathcal G \boxtimes \mathcal O_{\mathbf P^{d-2}}(-1)$ and $F:=\mathcal H^{\ast}\boxtimes \mathcal O_{\mathbf P^{d-2}}$. These are two vector bundles on $\Pic^3(C)\times \mathbf P^{d-2}$. The rank of $E$ is equal to $d-4$, and the rank of $F$ is equal to $2$.\\
The multiplication of fibers $$H^0(\mathcal O_C(H-D))\otimes H^0(\mathcal O_C(D))\to H^0(\mathcal O_C(H))$$ induces a map of vector bundles $\Phi:E\to F$.
Set
$$X_1:=X_1(\Phi):=\{x\in \Pic^3(C)\times \mathbf P^{d-2}\vert \rk (\Phi_x)\leq 1\}$$
Consider the two projections

\vspace{0.3cm}

\xymatrix{
&&&&&\Pic^3(C) \times \mathbf P^{d-2}\ar[dr]^{p_2}\ar[dl]_{p_1}&\\
&&&&\Pic^3(C)&&\mathbf P^{d-2}.\\
}

\vspace{0.3cm}

\noindent We have the following:
\begin{itemize}
\item[(i)] Over every point $[\mathcal O_C(D)]\in \Pic^3(C)$ the fiber of $p_1\vert X_1$ is a $3$-dimensional rational normal scroll $V_{\vert D\vert}\subseteq [\mathcal O_C(D)]\times \mathbf P^{d-2} \cong \mathbf P^{d-2}$.
\item[(ii)] The image of such a fiber under the projection $p_2$ is thus the rational normal scroll $V_{\vert D\vert}$ in $\mathbf P^{d-2}$.
\item[(iii)] Consequently, $p_2(X_1)$ is the union of all $g^1_3(C)$-scrolls $V_{\vert D\vert}$ in $\mathbf P^{d-2}$ which again is equal to $\Sec_3(C)$.
\end{itemize}

\noindent Set $x_1$ to be the class of $X_1$. From the above we have $(p_2)_{\ast}(x_1)=[\Sec_3(C)]$. Let $h^{\prime}\subseteq \mathbf P^{d-2}$ be a hyperplane class and set $h:=(p_2)^{\ast}(h^{\prime})\subseteq \Pic^3(C)\times \mathbf P^{d-2}$.\\
Since $\Sec_3(C)\subseteq \mathbf P^{d-2}$ has dimension $5$, we obtain the degree of $\Sec_3(C)$ by intersecting with $(h^{\prime})^5$.\\
Now we have the following:

\vspace{-0.3cm}

\begin{eqnarray*}
\deg(\Sec_3(C))&=&[Sec^3(C)].(h^{\prime})^5=(p_2)_{\ast}(x_1).(h^{\prime})^5\\
&=&(p_2)_{\ast}(x_1.p_2^{\ast}(h^{\prime})^5)=(p_2)_{\ast}(x_1.h^5)=x_1.h^5.\\
\end{eqnarray*}

\vspace{-0.2cm}

\noindent That is, now we have to find the class $x_1$ of $X_1(\Phi)$.\\
Since $X_1(\Phi)$ has expected dimension $5=\dim(\Pic^3(C)\times \mathbf P^{d-2})-(d-4-1)(2-1)$, by Porteous' formula (\cite{ACGH}, Chapter II, (4.2)) we obtain the following:

\vspace{-0.2cm}

\begin{eqnarray*}
x_1&=&\Delta_{1,d-5}(c_t(F-E))\\
&&\\
&=&\det \underbrace{\left(\begin{array}{cccccc}
c_1&c_2&c_3&\cdots&c_{d-6}&c_{d-5}\\
1&c_1&c_2&\cdots&c_{d-7} &c_{d-6}\\
0&1&c_1&\cdots &c_{d-8}&c_{d-7}\\
&&&\ddots&\vdots&\vdots\\
0&0&0&\cdots&c_1&c_2\\
0&0&0&\cdots&1&c_1\\
\end{array}\right)}_{=:\mathcal A_{d-5}},
\end{eqnarray*}

\noindent where $c_i:=c_i(F-E)$, and $c_i(F-E)$ is defined via $c_t(F-E):=\frac{c_t(F)}{c_t(E)}$.\\
Again we use $\Theta_{P,Q}$ as defined in Section \ref{preliminaries} when we talk about $\Theta$ on $\Pic^3(C)\cong \Jac(C)$.\\
Equations (\ref{cpolyH}) and (\ref{cpolyG}) gave us the following Chern polynomials:
$$c_t(\mathcal H)=c_t(\mathcal G)=e^{-\Theta t}.$$
We thus obtain
$$c_t(F)=c_t(p_1^{\ast}\mathcal H^{\ast})=c_{-t}(p_1^{\ast}\mathcal H)=e^{p_1^{\ast}\Theta t}.$$
We compute $c_t(E)$:\\
Let $\alpha_i$ be the Chern roots of $\mathcal G$, i.e. $c_t(\mathcal G)=\prod_{i=1}^{d-4}{(1+\alpha_i t)}$, and set $\beta_i:=p_1^{\ast}(\alpha_i)$. 

\noindent Then we obtain the following:

\begin{eqnarray*}
c_t(E)&=&\prod_{i=1}^{d-4}\left(1+\left(\beta_i-h\right)t\right)=\prod_{i=1}^{d-4}\left(1-ht\right)\left(1+\beta_i \frac{t}{1-ht}\right)\\
&=&(1-ht)^{d-4}\prod_{i=1}^{d-4}\left(1+\beta_i \frac{t}{1-ht}\right)=(1-ht)^{d-4}c_{\frac{t}{1-ht}}\left(p_1^{\ast}\mathcal G\right)\\
&=&(1-ht)^{d-4}e^{\frac{-p_1^{\ast}\Theta t}{1-ht}}.\\
\end{eqnarray*}

\noindent In the following we will identify $\Theta$ with $p_1^{\ast}(\Theta)$, it will be clear from the context if we mean $\Theta$ on $\Pic^3(C)$ or $\Theta$ on $\Pic^3(C)\times \mathbf P^{d-2}$.\\
We conclude now: 

\vspace{-0.2cm}

\begin{eqnarray*}
c_t(F-E)&=&e^{\Theta t}(1-ht)^{4-d}e^{\frac{\Theta t}{1-ht}}=(1-ht)^{4-d}e^{\frac{2\Theta t-\Theta.h t^2}{1-ht}}\\
&=&(1-ht)^{4-d}\sum_{j=0}^{\infty}{\frac{1}{j!}(2\Theta t-\Theta.h t^2)^j(1-ht)^{-j}}\\
&=&\sum_{j=0}^{\infty}{(1-ht)^{4-d-j}\frac{1}{j!}(2\Theta t-\Theta.h t^2)^j}.
\end{eqnarray*}

\noindent Since $\Theta^3=0$, we only get some contribution from $j=0,1,2$ and thus obtain the following:

\vspace{-0.2cm}

\begin{eqnarray*}
c_t(F-E)&=&(1-ht)^{4-d}+(1-ht)^{3-d}(2\Theta t-\Theta.ht^2)\\
&+&\frac{1}{2}(1-ht)^{2-d}(4\Theta^2 t^2-4\Theta^2.ht^3+\Theta^2.h^2t^4)\\
&=&\sum_{k=0}^{\infty}\binom{d+k-3}{k}h^kt^{k}\\
&+&\sum_{k=0}^{\infty}\binom{d+k-3}{k}(2\Theta.h^k-2h^{k+1})t^{k+1}\\
&+&\sum_{k=0}^{\infty}\binom{d+k-3}{k}(2\Theta^2.h^k-3\Theta.h^{k+1}+h^{k+2})t^{k+2}\\
&+&\sum_{k=0}^{\infty}\binom{d+k-3}{k}(\Theta.h^{k+2}-2\Theta^2.h^{k+1})t^{k+3}\\
&+&\sum_{k=0}^{\infty}\frac{1}{2}\binom{d+k-3}{k}\Theta^2.h^{k+2}t^{k+4}.
\end{eqnarray*}

\noindent This implies that

\begin{eqnarray*}
c_i(F-E)&=&\left(\binom{d-3+i}{i}-2\binom{d-4+i}{i-1}+\binom{d-5+i}{i-2}\right)h^i\\
&+&\left(2\binom{d-4+i}{i-1}-3\binom{d-5+i}{i-2}+\binom{d-6+i}{i-3}\right)\Theta.h^{i-1}\\
&+&\left(2\binom{d-5+i}{i-2}-2\binom{d-6+i}{i-3}+\frac{1}{2}\binom{d-7+i}{i-4}\right)\Theta^2.h^{i-2}\\
&=&\binom{d-5+i}{i}h^i\\
&+&\left(\binom{d-5+i}{i-1}+\binom{d-6+i}{i-1}\right)\Theta.h^{i-1}\\
&+&\left(2\binom{d-6+i}{i-2}+\frac{1}{2}\binom{d-7+i}{i-4}\right)\Theta^2.h^{i-2}.
\end{eqnarray*}

\vspace{0.1cm}

\noindent Now the last step in the computation of $x_1$ is to find the determinant of the matrix $\mathcal A_{d-5}$:
\begin{prop}\label{determinant}
Set $c_i:=c_i(F-E)$, where $c_i(F-E)$ is defined via $c_t(F-E):=\frac{c_t(F)}{c_t(E)}$.\\
For $d\geq 8$ the determinant of the matrix

$$
\mathcal A_{d-5}=\left(\begin{array}{cccccc}
c_1&c_2&c_3&\cdots&c_{d-6}&c_{d-5}\\
1&c_1&c_2&\cdots&c_{d-7} &c_{d-6}\\
0&1&c_1&\cdots &c_{d-8}&c_{d-7}\\
&&&\ddots&\vdots&\vdots\\
0&0&0&\cdots&c_1&c_2\\
0&0&0&\cdots&1&c_1\\
\end{array}\right)
$$

\noindent is equal to 

\begin{eqnarray*}
\mathcal D_{d-5}&=&\left(\frac{1}{2}\binom{d-2}{3}-(d-4)\right)\Theta^2.h^{d-7}\\
&+&\left(\binom{d-3}{2}-1\right)\Theta.h^{d-6}\\
&+&(d-4)h^{d-5}.\\
\end{eqnarray*}
\end{prop}

\begin{proof}
Let $d$ be fixed. Set $d_0:=1$ and for $n=1,\dots, d-5$, $k=2,\dots d-6$, set 

$$
d_{n}:=\det\left(\begin{array}{ccccc}
c_1&c_2&c_3&\cdots&c_{n}\\
1&c_1&c_2&\cdots&c_{n-1}\\
&&\ddots&\vdots&\vdots\\
0&0&\cdots&c_1&c_2\\
0&0&\cdots&1&c_1\\
\end{array}\right)
$$

\noindent and

$$
b_{n,k}:=\det\left(\begin{array}{cccccc}
c_k&c_{k+1}&c_{k+2}&\cdots&c_{n-1}&c_{n}\\
1&c_1&c_2&\cdots&c_{n-(k+1)}&c_{n-k}\\
0&1&c_1&\cdots&c_{n-(k+2)}&c_{n-(k+1)}\\
&&\ddots&\vdots&\vdots\\
0&0&0&\cdots&c_1&c_2\\
0&0&0&\cdots&1&c_1\\
\end{array}\right).
$$

\vspace{0.3cm}

\noindent By expansion with respect to the first column we have for each $n$ and $k$: 

\vspace{-0.3cm}

\begin{eqnarray*}
d_{n}&=&c_1d_{n-1}-b_{n,2}\\
\end{eqnarray*}

\vspace{-0.4cm}

\noindent and

\vspace{-0.4cm}

\begin{eqnarray*}
b_{n,k}&=&c_kd_{n-k}-b_{n,k+1}.\\
\end{eqnarray*}

\noindent This gives us by induction:

\begin{eqnarray*}
d_{n}&=&\sum_{i=1}^{n}{(-1)^{i-1}c_id_{n-i}}.\\
\end{eqnarray*}

\noindent Computing $d_n$ for low $n$ leads us to the following statement:

\begin{lemma}\label{dn}
For $n\geq 3$ we have
\begin{eqnarray*}
d_n&=&\binom{d-4}{n}h^n+\left(\binom{d-3}{n}-\binom{d-5}{n}\right)\Theta.h^{n-1}\\
&+&\left(\frac{1}{2}\binom{d-2}{n}-\binom{d-4}{n}+\frac{1}{2}\binom{d-6}{n}\right)\Theta^2.h^{n-2}.\\
\end{eqnarray*}
\end{lemma}

\newpage

\begin{proof}By induction over $n$:

\begin{eqnarray*}
d_n&=&\sum_{i=1}^n(-1)^{i-1}c_id_{n-i}\\
&=&\sum_{i=1}^n(-1)^{i-1}\binom{d-5+i}{i}\binom{d-4}{n-i}h^n\\
&+&\sum_{i=1}^n(-1)^{i-1}\Bigg(\binom{d-5+i}{i}\binom{d-3}{n-i}-\binom{d-5+i}{i}\binom{d-5}{n-i}\\
&&\qquad\qquad\quad+\binom{d-4}{n-i}\binom{d-6+i}{i-1}+\binom{d-4}{n-i}\binom{d-5+i}{i-1}\Bigg)\Theta.h^{n-1}\\
&+&\sum_{i=1}^n(-1)^{i-1}\Bigg(\frac{1}{2}\binom{d-5+i}{i}\binom{d-2}{n-i}-\binom{d-5+i}{i}\binom{d-4}{n-i}\\
&&\qquad\qquad\quad+\frac{1}{2}\binom{d-5+i}{i}\binom{d-6}{n-i}+\binom{d-6+i}{i-1}\binom{d-3}{n-i}\\
&&\qquad\qquad\quad-\binom{d-6+i}{i-1}\binom{d-5}{n-i}+\binom{d-5+i}{i-1}\binom{d-3}{n-i}\\
&&\qquad\qquad\quad-\binom{d-5+i}{i-1}\binom{d-5}{n-i}+2\binom{d-6+i}{i-2}\binom{d-4}{n-i}\\
&&\qquad\qquad\quad+\frac{1}{2}\binom{d-7+i}{i-4}\binom{d-4}{n-i}
\Bigg)\Theta^2.h^{n-2}.\\
\end{eqnarray*}

\noindent Using the binomial identities 

\begin{itemize}
\item[(a)] Upper negation: $\binom{-r}{m}=(-1)^m\binom{r+m-1}{m}$ for $r,m\in \mathbf N$,
\item[(b)] Vandermonde's identity: $\sum_{k=0}^{r}\binom{m}{k}\binom{s}{r-k}=\binom{m+s}{r}$ for $m,r,s\in \mathbf N$
\end{itemize}

\noindent we obtain the formula for $d_n$ as given in Lemma \ref{dn}.
\end{proof}

\noindent To finish the proof of Proposition \ref{determinant} we use Lemma \ref{dn} taking $n=d-5\geq 3$:

\vspace{-0.4cm}

\begin{eqnarray*}
\mathcal D_{d-5}&=&d_{d-5}=\binom{d-4}{d-5}h^{d-5}\\
&+&\Bigg(\binom{d-3}{d-5}-\binom{d-5}{d-5}\Bigg)\Theta.h^{d-6}\\
&+&\Bigg(\frac{1}{2}\binom{d-2}{d-5}-\binom{d-4}{d-5}+\frac{1}{2}\binom{d-6}{d-5}\Bigg)\Theta^2.h^{d-7}\\
&=&(d-4)h^{d-5}\\
&+&\Bigg(\binom{d-3}{2}-1\Bigg)\Theta.h^{d-6}\\
&+&\Bigg(\frac{1}{2}\binom{d-2}{3}-(d-4)\Bigg)\Theta^2.h^{d-7}.
\end{eqnarray*}
\end{proof}

\noindent Now we are able to deduce the formula for the degree of $\Sec_3(C)$ where $C$ is a curve of genus $2$ and degree $d\geq 8$ in $\mathbf P^{d-2}$:

\begin{prop}
The degree of the third secant variety $\Sec_3(C)$ of a smooth curve of genus $2$ and degree $d\geq 8$ in $\mathbf P^{d-2}$ is equal to
$$\binom{d-2}{3}-2(d-4).$$
\end{prop}

\begin{proof}
Since $\Sec_3(C)$ has dimension $5$, we have to intersect with $(h^{\prime})^5$ where $h^{\prime}$ is a hyperplane class in $\mathbf P^{d-2}$ in order to obtain the degree of $\Sec_3(C)$. From the above remarks we now have to find $\deg x_1.h^5$, where $h=(p_2)^{\ast}(h^{\prime})$.\\
We have 
\begin{eqnarray*}
\deg x_1.h^5&=&\deg \mathcal D_{d-5}.h^5\\
&=&\deg\Bigg(\frac{1}{2}\binom{d-2}{3}-(d-4)\Bigg)\Theta^2.h^{d-2}.\\
\end{eqnarray*}

\noindent Since $\deg\Theta^2.h^{d-2}=2$, (cf. Proposition \ref{theta2}) we finally obtain
$$\deg\mathcal D_{d-5}.h^{5}=2\Bigg(\frac{1}{2}\binom{d-2}{3}-(d-4)\Bigg)=\binom{d-2}{3}-2(d-4).$$
\end{proof}

\subsection*{Acknowledgements}
This paper originated from a part of my Ph.D. thesis. I wish to thank my advisor Kristian Ranestad for continuous helpful advice.

\end{document}